\documentclass[11pt,a4paper]{amsart}
\pdfoutput=1
\usepackage{times}
\usepackage{amssymb}

\usepackage{ifthen}
\usepackage{cite}
\usepackage{graphicx}

\makeatletter

\newcommand{\gb}{\beta}
\newcommand{\gd}{\delta}
\newcommand{\gep}{\epsilon}

\newcommand{\gk}{\kappa}
\newcommand{\gl}{\lambda}
\newcommand{\go}{\omega}
\newcommand{\gs}{\sigma}

\newcommand{\gD}{\Delta}

\newcommand{\gO}{\Omega}

\newcommand{\cA}{\mathcal{A}}
\newcommand{\cB}{\mathcal{B}}

\newcommand{\cD}{\mathcal{D}}
\newcommand{\cE}{\mathcal{E}}
\newcommand{\cF}{\mathcal{F}}
\newcommand{\cG}{\mathcal{G}}
\newcommand{\cH}{\mathcal{H}}
\newcommand{\cI}{\mathcal{I}}
\newcommand{\cJ}{\mathcal{J}}

\newcommand{\cL}{\mathcal{L}}

\newcommand{\1}{1}

\newcommand{\C}{\mathbb{C}}

\newcommand{\N}{\mathbb{N}}

\newcommand{\R}{\mathbb{R}}

\newcommand{\p}{\partial}

\newtheorem{theorem}{Theorem}[section]

\newtheorem{lemma}[theorem]{Lemma}
\newtheorem{corollary}[theorem]{Corollary}

\newtheorem{definition}[theorem]{Definition}

\theoremstyle{remark}

\newtheorem{remark}[theorem]{Remark}

\newcommand{\diag}[1]{\text{diag}(#1)}
\newcommand{\erf}{\mathop{\operator@font erf}\nolimits}
\newcommand{\erfc}{\mathop{\operator@font erfc}\nolimits}
\newcommand{\sign}{\mathop{\operator@font sign}\nolimits}

\newenvironment{enum_i}
    {\begin{enumerate}}
    {\end{enumerate}}

\newif\if@golden  \@goldentrue
\newcommand{\f@ctor}{1}
\newlength{\aiv@width}  
\newlength{\aiv@height} 
\newlength{\tmp@width}  
\newlength{\tmp@height} 
\if@twocolumn\if@golden
  \else\fi
\else\if@golden\ifcase\@ptsize\relax\or
\or\fi
  \else\ifcase\@ptsize\relax\or
\or\fi\fi
\if@twocolumn\if@golden\ifcase\@ptsize\relax\renewcommand{\f@ctor}{53}
  \or\renewcommand{\f@ctor}{46}\or\renewcommand{\f@ctor}{43}\fi
  \else\ifcase\@ptsize\relax\renewcommand{\f@ctor}{51}\or
  \renewcommand{\f@ctor}{45}\or\renewcommand{\f@ctor}{42}\fi\fi
\else\if@golden\ifcase\@ptsize\relax \renewcommand{\f@ctor}{46}
  \or\renewcommand{\f@ctor}{43}\or\renewcommand{\f@ctor}{43}\fi
  \else\ifcase\@ptsize\relax\renewcommand{\f@ctor}{43}
  \or\renewcommand{\f@ctor}{40}\or\renewcommand{\f@ctor}{40}\fi\fi\fi
\multiply\textheight by \f@ctor
\let\comp\circ

\newcommand{\op}{\ensuremath^\circ}

\newcommand{\cGo}{\ensuremath\cG\op}

\newcommand{\sgl}{\ensuremath\sqrt{2\gl}}

\newcommand{\da}{\ensuremath\downarrow}
\newcommand{\ua}{\ensuremath\uparrow}

\newcommand{\cond}{\ensuremath\,\big|\,}
\newcommand{\Co}{\ensuremath C_0(\cG)}
\newcommand{\Coo}{\ensuremath C_0^{0,2}(\cG)}
\newcommand{\Cii}{\ensuremath C_0^2(\cG)}
\newcommand{\Coe}{\ensuremath C_0(E)}
\newcommand{\Ra}{$\Rightarrow$\space}
\newcommand{\limep}{\ensuremath\lim_{\gep\da 0}}

\newcommand{\cGD}{\ensuremath \cG^\gD}

\newcommand{\Ieqref}[1]{\textup{\tagform@{I.\ref{I_#1}}}}

\newcommand{\IIeqref}[1]{\textup{\tagform@{II.\ref{II_#1}}}}

\newcommand{\abs}{\ensuremath ^{\text{\itshape abs}}}

\newlength{\BCs@ze}
\newlength{\BCsh@ft}
\ifcase\@ptsize\relax
\or
\or
\fi
\DeclareFixedFont\MT{OMS}{cmsy}{m}{n}{\BCs@ze}    
\newcommand{\BigCart}{\ensuremath\mathop{\raisebox{\BCsh@ft}{{\MT\char"02}}}}

\def\pdftitle{\@gobble}
\let\setdif\setminus

\makeatother

\numberwithin{equation}{section}
\allowdisplaybreaks[2]
\date{December 6, 2010}

\title[Brownian Motions on Metric Graphs]{%
Brownian Motions on Metric Graphs I\\
{\small Definition, Feller Property, and Generators}
}

\pdftitle{Brownian Motions on Metric Graphs I -
          Definition, Feller Property, and Generators}

\author[V.~Kostrykin]{Vadim Kostrykin}
\address{Vadim Kostrykin\newline
Institut f\"ur Mathematik\newline
Johannes Gutenberg--Universit\"at\newline
D--55099 Mainz, Germany}
\email{kostrykin@mathematik.uni-mainz.de}

\author[J.~Potthoff]{J\"urgen Potthoff}
\address{J\"urgen Potthoff\newline
Institut f\"ur Mathematik, Universit\"at Mannheim\newline
D--68131 Mann\-heim, Germany}
\email{potthoff@math.uni-mannheim.de}

\author[R.~Schrader]{Robert Schrader}
\address{Robert Schrader\newline
Institut f\"{u}r Theoretische Physik\newline
Freie Universit\"{a}t Berlin, Arnimallee~14\newline
D--14195 Berlin, Germany}
\email{schrader@physik.fu-berlin.de}

\copyrightinfo{2010}{V.~Kostrykin, J.~Potthoff, R.~Schrader}

\begin{document}
\begin{abstract}
Brownian motions on a metric graph are defined, their Feller property is proved, and
their generators are characterized. This yields a version of Feller's theorem for
metric graphs. 
\end{abstract}

\maketitle
\thispagestyle{empty}

\section{Introduction} \label{sect_1}

In his pioneering articles~\cite{Fe52, Fe54, Fe54a}, Feller raised the problem of
characterizing and constructing all Brownian motions on a finite or on a
semi-infinite interval. In the sequel this problem stimulated very important
research in the field of stochastic processes, and the problem of constructing
all such Brownian motions found a complete solution~\cite{ItMc63, ItMc74} via the
combination of the theory of the local time of Brownian motion~\cite{Le48}, and the
theory of (strong) Markov processes~\cite{Bl57, Dy61, Dy65a, Dy65b, Hu56}.

On the other hand, there is a growing interest in metric graphs, that is, piecewise
linear varieties where the vertices may be viewed as singularities. Metric graphs
arise naturally as models in many domains, such as physics, chemistry, computer science
and engineering to mention just a few --- we refer the interested reader to~\cite{Ku04}
for a review of such models and for further references.

Therefore it is natural to extend Feller's problem to metric graphs. The
present paper is the first in a series of three articles~\cite{BMMG2, BMMG3} (together
with a more pedagogical one~\cite{BMMG0}, in which the well-known classical
cases of finite and semi-infinite intervals are revisited) on the characterization
and the construction of all Brownian motions on metric graphs. Stochastic processes,
in particular Brownian motions and diffusions, on locally one-dimensional
structures, notably on graphs and networks, have already been studied in a number of
articles, out of which we want to mention~\cite{BaCh84, DeJa93, EiKa96, FrSh00,
FrWe93, Fr94, Gr99, Kr95} in this context.

Heuristically, a metric graph $(\cG,d)$ can be thought of as the union of a
collection of finite or semi-infinite closed intervals which are glued together at
some of their endpoints which form the vertices of the graph, while the intervals
are its edges. The metric $d$ is then defined in the canonical way as the length of
a shortest path between two points along the edges, and the length along the edges
is measured as for usual intervals. For a more formal definition of metric graphs
see section~\ref{ssect_2_1}. We will only consider \emph{finite} graphs, that is,
those for which the sets of vertices and edges are finite. For the definition of a
Brownian motion on the metric
graph $(\cG,d)$ we take a standpoint similar to the one of Knight~\cite{Kn81} for
the semi-line or a finite interval: It is a strong Markov process with c\`adl\`ag
paths which are continuous up to the lifetime, and which up to the first passage time
at a vertex is a standard Brownian motion on the edge where it started. For the
formal definition, cf.\ section~\ref{sect_3}.

The crucial problem is then to characterize the behaviour of the stochastic process
when it reaches one of the vertices of the graph $\cG$, or in other words, the
characterization of the boundary conditions at the vertices of the Laplace operator
which generates the stochastic process. We want to mention in passing that in an
$L^2$-setting all boundary conditions for Laplace operators on $\cG$ which make them
self-adjoint operators have been characterized in~\cite{KoSc99, KoSc00}. In this
series of papers we shall work with the Banach space of continuous functions on
$\cG$ which vanish at infinity. The main result of the present paper is Feller's
theorem for metric graphs (cf.\ theorem~\ref{thm_5_3}). Roughly speaking it
states that all possible boundary conditions are \emph{local} boundary conditions of
\emph{Wentzell} type, i.e., linear combinations of the value of the function with
its first and second (directional) derivatives at each vertex, subject to certain
conditions on the coefficients. We want to emphasize here, that the fact that we
only obtain local boundary conditions is due to the assumption that the paths of the
Brownian motion have no jumps during their lifetime. On a more technical level this
assumption entails that we deal with Feller processes --- as is proved in
section~\ref{sect_4} --- which is of considerable advantage when we prove the strong
Markov property of the processes which we will construct in the follow-up
papers~\cite{BMMG2, BMMG3}.  On the other hand, in a forthcoming work we shall embed
the situation into the larger framework of Ray processes, and there we shall deal
also with non-local boundary conditions, allowing the processes to have jumps from the
vertices into $\cG$ in addition to the jumps to the cemetery point.

The paper is organized in the following way. In section~\ref{sect_2} we recall the
pertinent notions of (finite) metric graphs, and of strong Markov processes on metric graphs,
at the same time setting up our notation. In section~\ref{sect_3} Brownian motions
on metric graphs are defined, and some consequences of this definition are
discussed. The proof of the statement that Brownian motions on metric graphs are
Feller processes is given in section~\ref{sect_4}. Finally, we state and prove
Feller's theorem for metric graphs in section~\ref{sect_5}. In appendix~\ref{app_A}
we give a short account of Feller semigroups in a form which we find especially
convenient for the purposes of the present paper, but which we could not find in
this form elsewhere.

The contents of the other two papers in this series are as follows. In the
article~\cite{BMMG2} all Brownian motions on single vertex graphs (roughly speaking,
$n$ semi-lines $[0,+\infty)$ glued together at the origin) are constructed, and in
the article~\cite{BMMG3} these Brownian motions are pieced together pathwise to yield all
possible Brownian motions on a general metric graph.

\vspace{1.5\baselineskip}
\noindent
\textbf{Acknowledgement.}
The authors thank Mrs.~and Mr.~Hulbert for their warm hospitality at the
\textsc{Egertsm\"uhle}, Kiedrich, where part of this work was done.
J.P.\ gratefully acknowledges fruitful discussions with O.~Falkenburg,
A.~Lang and F.~Werner. R.S.~thanks the organizers of the \emph{Chinese--German
Meeting on Stochastic Analysis and Related Fields}, Beijing, May 2010,
where some of the material of this article was presented.

\section{Preliminary Definitions and Notations}  \label{sect_2}
\subsection{Metric Graphs} \label{ssect_2_1}
Throughout this paper we consider a fixed finite \emph{metric graph} $(\cG,d)$.
That is, $\cG$ is a quadruple $(V,\cI,\cE,\p)$, where $V$ is a finite set of
\emph{vertices}, $\cI$ is a finite set of \emph{internal edges}, $\cE$ is a finite
set of \emph{external edges}, and $\p$ is a map from the set $\cL=\cI\cup\cE$ of
\emph{edges} into $(V\times V)\cup V$, which maps an internal edge $i\in\cI$ to an
ordered pair $(\p^-(i),\p^+(i))\in V\times V$ of vertices, called the
\emph{initial} and \emph{final vertex of $i$}, while $e\in\cE$ is
mapped to $\p(e)\in V$, called the \emph{initial vertex of $e$}.  Every edge
$l\in\cL$ is assumed to be isometrically isomorphic to an interval $I_l$, namely
every $e\in\cE$ is in one-to-one correspondence with the half line $[0,+\infty)$,
while for every $i\in\cI$ there exists $\rho_i>0$ so that $i$ is isomorphic to
$[0,\rho_i]$. Under these isomorphisms, for $i\in\cI$ we have that $\p_-(i)$
corresponds to $0$ while $\p^+(i)$ corresponds to $\rho_i$, and for $e\in\cE$, the
vertex $\p(e)$ corresponds to $0$. $\rho_i$ is called the \emph{length}
of the internal edge $i\in\cI$. Moreover, we suppose that the edges are sets with an
ordering as induced by the isomorphisms mentioned above.
$\cL(v)=\{l\in\cL,\,v\in\partial(l)\}$ is the set of edges incident with $v$.
For notational simplicity we will also use $\p(l)$, $l\in\cL$, to denote the
set consisting of $\p^-(l)$ and $\p^+(l)$ if $l\in \cI$, and of $\p(l)$ if
$l\in\cE$.

In~\cite{KoSc06} the standard notion of a \emph{walk} on a graph (e.g., \cite{Ju05})
has been generalized to graphs of the above type, and therefore we obtain in a
natural way a metric $d$ on $\cG$ as the minimal length of all walks leading from one
point to another, where the length is measured along the edges as
induced by the isometry with the corresponding intervals.

In the sequel it will be convenient --- and without danger of confusion --- to
identify the abstract graph $\cG$ with its isomorphic \emph{geometric graph} (e.g.,
\cite{Ju05}). In other words, we also consider $\cG$ as a union of the intervals corresponding
to the edges, subject to the equivalence relation defined by the combinatorial
structure of the graph which identifies those endpoints of the intervals which
correspond to vertices to which the respective edges are incident. Similarly,
we shall often identify the edges with the intervals they are isomorphic to.

Clearly, $(\cG,d)$ is a complete, separable metric space, and hence it is a Polish
space. The Borel $\gs$--algebra of $\cG$ is denoted by $\cB(\cG)$. We write $B_r(\xi)$
for the open ball with radius $r>0$ and center $\xi\in\cG$.

For $l\in\cL$, $l\op$ denotes the open interior of $l$, i.e., $l\op$ is the subset
of $l$ being isomorphic to $(0,\rho_i)$ if $l=i\in\cI$, and to $(0,+\infty)$ if
$l\in\cE$. We set $\cGo = \cG\setdif V$ to be the interior of $\cG$, and hence $\cG\op$
is the pairwise disjoint union of the open edges $l\op$, $l\in\cL$. In particular, every
$\xi\in \cGo$ is in one-to-one correspondence with its \emph{local coordinate}
$(l,x)$, $l\in\cL$, $x\in I_l\op$, and we may and will write $\xi=(l,x)$.

Assume that $f$ is a real valued function on $\cG$. Then $f$ is in one-to-one
correspondence with the family of functions $(f_l,\,l\in\cL)$ where $f_l$ is the
restriction of $f$ to the edge $l\in\cL$. (Of course, if $v\in V$ is a vertex with
which the edges $l$, $l'\in\cL$ are incident, then we have to have
$f_l(v)=f_{l'}(v)$.) Sometimes it will also be convenient to write $f_l(x)$ instead
of $f(\xi)$, for $\xi\in\cGo$ having local coordinates $(l,x)$.

The space of real valued, bounded measurable functions will be denoted by $B(\cG)$,
while $C_0(\cG)$ denotes the space of continuous real valued functions on $\cG$ which
vanish at infinity. Both spaces are equipped with the sup-norm, denoted by
\mbox{$\|\cdot\|$}. \mbox{$(B(\cG),\|\cdot\|)$} and \mbox{$(C_0(\cG),\|\cdot\|)$} are
Banach spaces, the latter being separable.

\subsection{Markov Processes on Metric Graphs}  \label{ssect_2_2}
Let $(\cG,d)$ be a metric graph as in the previous subsection. Furthermore let
$(\gO,\cA)$ be a measurable space, equipped with a family $(P_\xi,\, \xi\in\cG)$ of
probability measures. The expectation with respect to $P_\xi$, $\xi\in\cG$, will be
denoted by $E_\xi(\,\cdot\,)$. We will say that a statement holds \emph{almost
surely} (a.s.), if for all $\xi\in\cG$ the statement holds almost surely with
respect to $P_\xi$.

Let $\gD$ be a point not in $\cG$ which we will view as a cemetery point. By $\cGD$
we denote the union $\cG\cup\{\gD\}$, where $\gD$ is adjoined to $\cG$ as an isolated
point. We define the $\gs$--algebra $\cB(\cGD)$ on $\cGD$ as the obvious minimal extension
of $\cB(\cG)$. All real valued functions $f$ on $\cG$ are understood as being extended
to $\cGD$ with $f(\gD)=0$.

We consider a $\cGD$--valued normal homogeneous Markov process $X = (X_t,\,t\ge 0)$
on $(\gO,\cA)$ relative to a filtration $\cF=(\cF_t,\,t\ge 0)$ in $\cA$ with
c\`adl\`ag paths. Throughout, we will suppose that the filtration $\cF$ is right
continuous and complete for the family $(P_\xi,\, \xi\in\cG)$, that is, for all
$t\ge 0$, $\cF_t=\cap_{\gep>0}\cF_{t+\gep}$, and $\cF_0$ contains all subsets of
$\gO$ which are negligible for all $P_\xi$, $\xi\in \cG$. Also $\gD$ is a cemetery
state for $X$, i.e., almost surely $X_s = \gD$, $s\ge 0$, entails $X_t = \gD$ for all
$t\ge s$. The lifetime $\zeta$ of $X$ is defined by $\zeta = \inf\{t\ge0,\,X_t =
\gD\}$.

As in~\cite{ReYo91} we assume that the transition probabilities of $X$ are
given in terms of a transition function $P = (P_t,\,t\ge 0)$, i.e.,
\begin{equation*}
    P_\xi(X_t\in C) = P_t(\xi,C),\qquad \xi\in\cG,\,t\ge 0,\,C\in\cB(\cGD).
\end{equation*}
In particular, for all $t\ge 0$, $C\in\cB(\cGD)$, the mapping $\xi\mapsto
P_\xi(X_t\in C)$ is measurable. In terms of the transition function $P$ the Markov
property of $X$ can be written as follows:
\begin{equation*}
    P_\xi(X_t\in C\cond \cF_s)
        = P_{t-s}(\xi,C), \qquad 0\le s\le t,\,C\in\cB(\cGD),\,\xi\in\cG.
\end{equation*}

It will be convenient and there is no loss of generality to assume the existence of a
shift operator $\theta: \R_+\times \gO\rightarrow \gO$, $(s,\go)\mapsto
\theta_s(\go)$, so that a.s.\ for all $t$, $s\ge 0$,
\begin{equation}    \label{eq_2_1}
    X_t \comp \theta_s = X_{t+s}.
\end{equation}

Thus if $X$ is a strong Markov process with respect to $\cF$, its strong Markov
property can be expressed in the following way. Let $S$ be an $\cF$--stopping time,
and as is usual denote $\gs$--algebra of the past of $S$ by $\cF_S$. If $\xi\in\cG$,
and $Z$ is a positive or bounded random variable, then
\begin{equation}    \label{eq_2_2}
    E_\xi\bigl(Z\comp \theta_S\cond \cF_S\bigr) = E_{X_S}(Z),
\end{equation}
holds $P_\xi$--a.s.\ on the set $\{X_S\ne \gD\} = \{S<\zeta\}$.

For a subset $A$ of $\cG$ we shall denote its hitting time by $X$ by $H_A$,
\begin{equation*}
    H_A = \inf\{t>0,\,X_t\in A\}
\end{equation*}
and if $A=\{\xi\}$, $\xi\in\cG$, we simply write $H_\xi$.

We shall occasionally take the liberty to write $X(t)$ for $X_t$, $t\in\R_+$, or
$H(A)$ for $H_A$, $A\subset \cG$, whenever it is typographically more convenient.

The semigroup $U=(U_t,\,t\ge 0)$ associated with $X$ and acting linearly on $B(\cG)$
is defined by
\begin{equation}    \label{eq_2_3}
    U_t f(\xi) = E_\xi\bigl(f(X_t)\bigr) = \int_\cG f(\eta)\,P_t(\xi,d\eta),
\end{equation}
for $f\in B(\cG)$, $\xi\in\cG$. The bound $\|U_t f\|\le \|f\|$ holds for all $f\in
B(\cG)$, $t\ge 0$. The resolvent $R=(R_\gl,\,\gl>0)$ associated with $X$, and acting on
$B(\cG)$, is defined by
\begin{equation}    \label{eq_2_4}
    R_\gl f(\xi) = \int_0^\infty e^{-\gl t} U_t f(\xi)\,dt,\qquad \gl>0,
\end{equation}
and satisfies
\begin{equation}    \label{eq_2_5}
    \bigl\|R_\gl f\bigr\| \le \frac{1}{\gl}\,\|f\|.
\end{equation}
We shall denote the restrictions of the semigroup $U$ and the resolvent $R$ to
$\Co$ by the same symbols.

Assume that $X$ is a strong Markov process with respect to the filtration $\cF$. A
direct consequence of the strong Markov property is the \emph{first passage time
formula} for the resolvent (e.g., \cite{Ra56} or \cite{ItMc74}): Let $\xi\in\cG$, $f\in
B(\cG)$, $\gl>0$, and assume that $S$ is a $\cF$--stopping time which is $P_\xi$--a.s.\
finite. Then
\begin{equation} \label{eq_2_6}
    R_\gl f(\xi) = E_\xi\Bigl(\int_0^S e^{-\gl t} f(X_t)\,dt\Bigr)
                    + E_\xi\bigl(e^{-\gl S}\,R_\gl f(X_S)\bigr)
\end{equation}
holds true.

We shall often have occasion to use a standard Brownian family on the real line $\R$
as a family of reference processes: Let $(\gO',\cA')$ denote another measurable
space with a family $(Q_x,\,x\in\R)$ of probability measures, and for every
$x\in\R$, $(B_t,\,t\in\R_+)$ a standard Brownian motion on $\R$ starting
$Q_x$--a.s.\ in $x$. Expectations with respect to $Q_x$ will be denoted by
$E^Q_x(\,\cdot\,)$. The Brownian family is equipped with a filtration denoted by
$\cJ=(\cJ_t,\,t\ge 0)$, and throughout we assume --- as we may --- that $\cJ$ is
right continuous and complete for the family $(Q_x,\,x\in\R)$. (For example, we can
always consider the natural filtration generated by $(B_t,\,t\in\R_+)$, and then
choose its universal augmentation, e.g., \cite[Chapter~III.2]{ReYo91} or
\cite[Chapter~2.7]{KaSh91}.) $H_A^B$ denotes the hitting time of the set
$A\subset\R$ by $B$, and as above we simply write $H^B_x$ for $H^B_{\{x\}}$, and
occasionally $B(t)$ for $B_t$, $H^B(A)$ for $H^B_A$, $A\subset\R$.

Consider an edge $l\in\cL$, and the interval $I_l\subset \R_+$ that $l$ is
isomorphic to. We set $\p(I_l)=\{0,\rho_l\}$ if $l\in\cI$, and $\p(I_l)=\{0\}$
if $l\in\cE$. For $t\in\R_+$ set
\begin{equation*}
    B_t^l =   B\bigl(t\land H^B_{\p(I_l)}\bigr)
\end{equation*}
where $s\land t = \min\{s,t\}$, $s$, $t\in\R_+$, i.e., $B^l$ is a Brownian motion on
$\R$ with absorption in the endpoint(s) of $I_l$. Under the family $(Q_x,\,x\in
I_l)$ we call this process the \emph{absorbed Brownian motion on $I_l$}.

\section{Definition of Brownian Motions on a Metric Graph}   \label{sect_3}
In analogy with~\cite[Chapter~6]{Kn81} we define a Brownian motion on a metric graph
$\cG$ as follows.

\begin{definition}  \label{def_3_1}
A \emph{Brownian motion on $\cG$} is a normal strong Markov process $X=(X_t,\,t\ge
0)$ with state space $\cGD$ and lifetime $\zeta$. The sample paths of $X$ are right
continuous with left limits in $\cG$, and they are continuous on $[0,\zeta)$.
Furthermore, for every $\xi\in\cGo$ with local coordinates $(l,x)$, $l\in\cL$, $x\in
I_l\op$, the process $X\abs = \bigl(X(t\land H_V),\,t\ge 0\bigr)$, with start in
$\xi$ and absorption in the set of vertices $V$, is equivalent to an absorbed Brownian
motion on $I_l$ with start in $x$.
\end{definition}

\begin{remark}  \label{rem_3_2}
According to our convention of subsection~\ref{ssect_2_2}, we consider a Brownian
motion $X$ on $\cG$ as a strong Markov process with respect to a filtration
$\cF=(\cF_t,\,t\ge 0)$ which is right continuous and complete.
\end{remark}

\begin{remark}	\label{rem_3_2a}
Suppose that $\xi\in l$, $l\in\cL$, then $H_V = H_{\p(l)}$, because
paths starting at $\xi$ hit the set $\p(l)$ before any other
vertex due to the continuity assumption.
\end{remark}

Consider a Brownian motion $X$ on $\cG$. Let us state the last condition in
definition~{\ref{def_3_1}} more explicitly. Fix $l\in\cL$, and $\xi\in l\op$ with
local coordinates $(l,x)$. Then for all $n\in\N$, $t_1$, \dots, $t_n\in\R_+$, with
$t_1\le t_2\le\dotsb\le t_n$, and all $A_1$, \dots, $A_n$ in the Borel
$\gs$--algebra $\cB(l)$ of~$l$,
\begin{equation}    \label{eq_3 1}
\begin{split}
    P_\xi\bigl(X_{t_1}\in A_1,&\dotsc,X_{t_n}\in A_n, t_n\le H_V\bigr)\\
        &= Q_x\bigl(B_{t_1}\in A_1,
                \dotsc,B_{t_n}\in A_n, t_n\le H^B_{\p(I_l)}\bigr).
\end{split}
\end{equation}
For simplicity we have identified the set $A_i \subset l$ with its isomorphic image
in $I_l$. Observe that in particular under $P_\xi$, the
stopping time $H_V = H_{\p(l)}$ has the same law as $H^B_{\p(I_l)}$ under $Q_x$, and
especially we get that for all $\xi\in\cG$, $P_\xi(H_V<+\infty)=1$.

It follows from definition~\ref{def_3_1} that any discontinuity of the paths of $X$
can only occur at the vertices of $\cG$, and it consists in a jump to the cemetery
state $\gD$. Hence if the process starts in $\xi\in\cGo$, it cannot reach the
cemetery state $\gD$ before hitting $V$. On the other hand, if the process starts in
$v\in V$, $P_v$--a.s.\ it cannot jump right away to $\gD$, because this would
contradict the right continuity of the paths and the requirement $P_\xi(X_0=\xi)=1$
for all $\xi\in \cG$. In particular, we have for all $\xi\in\cG$, $P_\xi(\zeta\ge
H_V)=1$.

For the following discussion we assume that the process $X$ starts at a vertex $v\in
V$, and consider the exit time from $v$, i.e., the stopping time $S_v =
H(\cGo)$. It is well known (e.g., \cite{Kn81, ReYo91, DyJu69}) that because of the
strong Markov property of $X$, $S_v$ is under $P_v$ exponentially distributed with a
rate $\gb_v\in[0,+\infty]$. Thus there are three possibilities:

\vspace{.5\baselineskip}\noindent
\emph{Case $\gb_v=0$}: In this case the process stays at $v$ forever, i.e., $v$ is a
\emph{trap}, and the process is given by $(X(t\land H_v),\,t\in\R_+)$.

\vspace{.5\baselineskip}\noindent
\emph{Case $0<\gb_v<+\infty$}: In this case the process stays at $v$ $P_v$--a.s.\ for
a strictly positive, finite moment of time, i.e., $v$ is \emph{exponentially
holding}. It is well known (cf., e.g., \cite[p.~154]{Kn81},
\cite[p.~104, Prop.~3.13]{ReYo91}) that then
the process has to leave $v$ by a jump, and by our assumption of path continuity on
$[0,\eta)$, the process has to jump to the cemetery $\gD$.

\vspace{.5\baselineskip}\noindent
\emph{Case $\gb_v=+\infty$}: In this case the $X$ leaves the vertex $v$ immediately,
and it begins a Brownian excursion into one of the edges incident with the vertex
$v$.

\section{Feller Property}    \label{sect_4}
In this section we prove that the semigroup $U$ associated with every Brownian
motion on $\cG$ has the Feller property (see, e.g., \cite{Kn81, ReYo91}, or
definition~\ref{def_A_1} in appendix~\ref{app_A}).

We begin with the following simple lemma.

\begin{lemma}   \label{lem_4_1}
Assume that $\xi\in\cGo$. Then $P_\xi$--a.s.\ $H_\eta$ converges to zero, whenever
$\eta$ increases or decreases to $\xi$ along the edge to which $\xi$ belongs.
\end{lemma}

\begin{proof}
First we notice that because up to time $\zeta$ the paths of $X$ are continuous,
$\eta\mapsto H_\eta$ is $P_\xi$--a.s.\ monotone decreasing as $\eta$ increases or
decreases to $\xi$. Therefore it is enough to show that $H_\eta$ converges to zero
in $P_\xi$--probability.

Let $\xi\in l\op$, $l\in\cL$, with local coordinates $(l,x)$, $x\in I_l\op$. Fix
$\gep>0$ small enough so that $(l,x\pm \gep)\in l\op$. Without loss of generality
we may assume that $d(\xi,\eta)<\gep$. We consider first the case where $\eta\da \xi$,
i.e., for the local coordinates $(l,y)$, $y\in I_l$, of $\eta$ we have $y\da x$.
Let $\gd>0$, and write
\begin{equation}    \label{eq_4_1}
    P_\xi(H_\eta>\gd)
        = P_\xi\bigl(H_\eta >\gd, H_{(l,x-\gep)}\ge H_\eta\bigr)
            + P_\xi\bigl(H_\eta >\gd, H_{(l,x-\gep)}< H_\eta\bigr).
\end{equation}
We estimate the second probability on the right hand side from above by
\begin{equation*}
P_\xi\bigl(H_{(l,x-\gep)}< H_\eta\bigr).
\end{equation*}
But this is the probability of the event that the process leaves the set on $l$
which in local coordinates is the interval $(l, [x-\gep,y])$ at the end point with
local coordinates $(l,x-\gep)$. Therefore this is an event which happens before the
process hits a vertex, and hence this probability is equal to the corresponding one
for a standard Brownian motion (e.g., \cite[Problem~6, p.~29]{ItMc74}):
\begin{equation*}
    P_\xi\bigl(H_{(l,x-\gep)}< H_\eta\bigl) = \frac{y-x}{y-x+\gep},
\end{equation*}
which converges to zero as $\eta\da \xi$. Similarly, the first probability on the right
hand side of equation~\eqref{eq_4_1} is equal to
\begin{align*}
    Q_x\bigl(H^B_y>\gd, H^B_{x-\gep}\ge H^B_y\bigr)
        &\le Q_x\bigl(H^B_y\ge \gd\bigr)\\[1ex]
        &= \int_\gd^\infty \frac{y-x}{\sqrt{2\pi t^3}}\,e^{-(y-x)^2/2t}\,dt\\[1ex]
        &= \erf\Bigl(\frac{y-x}{\sqrt{2\gd}}\Bigr),
\end{align*}
where we used the well-known density of $H^B_y$ under $Q_x$, e.g.,
\cite[p.~292]{Sc15}, \cite[section~1.7]{ItMc74}, or
\cite[Proposition~2.6.19]{KaSh91}. Clearly, the last expression converges to zero as
$y\da x$, i.e., as $\eta\da \xi$.

The case $\eta\ua \xi$ is treated in an analogous way.
\end{proof}

\begin{lemma}   \label{lem_4_2}
Let $\gl>0$, $v\in V$, and suppose that $l\in \cL(v)$. Then
\begin{equation}    \label{eq_4_2}
    \lim_{\eta\to v,\, \eta\in l} E_\eta\bigl(e^{-\gl H_v}\bigr) = 1.
\end{equation}
\end{lemma}

\begin{proof}
Fix $\gep>0$ in such a way that we have for every $v'\in V$, $v'\ne v$,
$d(v,v')>\gep$. We may assume that $d(v,\eta)<\gep$. Set
\begin{equation}    \label{eq_4_3}
    H_{v,\gep} = H\bigl(B_\gep(v)^c\bigr),
\end{equation}
where the superscript $c$ denotes the complement of a set. Write
\begin{equation}    \label{eq_4_4}
   1-E_\eta\bigl(e^{-\gl H_v}\bigr)
        = E_\eta\bigl(1-e^{-\gl H_v}; H_v \le H_{v,\gep}\bigr)
            + E_\eta\bigl(1-e^{-\gl H_v}; H_v > H_{v,\gep}\bigr),
\end{equation}
with the notation
\begin{equation}
    E_\eta(Z;C) = E_\eta(Z\,1_C)
\end{equation}
for positive or $P_\eta$--integrable random variables $Z$, and $C\in\cA$. Consider the
case where the vertex $v$ corresponds to the point with local coordinates $(l,0)$, the
case where $v$ corresponds to $(l,\rho_l)$ can be dealt with by an analogous argument.
Let $\eta$ have local coordinates $(l,y)$, $0\le y<\gep$. The second term on the right
hand side of equation~\eqref{eq_4_4} is less or equal to
\begin{align*}
    P_\eta(H_v>H_{v,\gep})
        &= Q_y(H_0^B > H_\gep^B)\\
        &= \frac{y}{\gep},
\end{align*}
which converges to zero with $y\da 0$, i.e., with $\eta\to v$. On the other hand
\begin{align*}
    E_\eta\bigl(1-e^{-\gl H_v}; H_v \le H_{v,\gep}\bigr)
        &=      E^Q_y\bigl(1-e^{-\gl H_0^B}; H_0^B \le H_\gep^B\bigr)\\
        &\le    E^Q_y\bigl(1-e^{-\gl H_0^B}\bigr)\\
        &=      1-e^{-\sgl y},
\end{align*}
where we used the well-known formula for the Laplace transform of the density of
$H_0^B$ under $Q_y$ (e.g., \cite[p.~26, eq.~5]{ItMc74}). Obviously this converges to
zero as $y\da 0$, i.e., as $\eta\to v$.
\end{proof}

\begin{theorem} \label{thm_4_3}
Every Brownian motion on $\cG$ is a Feller process.
\end{theorem}

\begin{proof}
The proof is based on the first passage time formula~\eqref{eq_2_6}. By
theorem~\ref{thm_A_3} in appendix~\ref{app_A} it suffices to show that for all
$\gl>0$, $R_\gl$ maps $\Co$ into itself, and that for all $f\in \Co$, $\xi\in\cG$,
$U_t f(\xi)$ converges to $f(\xi)$ as $t\downarrow 0$.

First we show that for every $\gl>0$, $R_\gl$ maps $\Co$ into itself. Assume that
$f\in \Co$. Consider the case $\xi\in\cGo$. Then it follows from lemma~\ref{lem_4_1}
as in~\cite[Section~3.6]{ItMc74} that  $R_\gl f$ is continuous at $\xi$. Consider now
the case $\xi=v\in V$, let $l$ belong to the set $\cL(v)$ of edges incident with $v$,
and let $\eta\in l$. Note that $P_\eta$--a.s.\ $H_v$ is finite (cf.\
section~\ref{sect_3}). Therefore we can employ equation~\eqref{eq_2_6} with $S=H_v$:
\begin{equation*}
    R_\gl f(\eta) = E_\eta\Bigl(\int_0^{H_v} e^{-\gl t} f(X_t)\,dt\Bigr)
                    + E_\eta\bigl(e^{-\gl H_v}\bigr)\,R_\gl f(v).
\end{equation*}
Using~\eqref{eq_2_5} we estimate in the following way
\begin{align*}
    \bigl|R_\gl f(\eta) - R_\gl f(v)\bigr|
        &\le \Bigl|E_\eta\Bigl(\int_0^{H_v} e^{-\gl t} f(X_t)\,dt\Bigr)\Bigr|\\
        &\hspace{4em} +\Bigl(1-E_\eta\bigl(e^{-\gl H_v}\bigr)\Bigr)\,\bigl|R_\gl f(v)\bigr|\\
        &\le \frac{2}{\gl}\,\|f\|\,\Bigl(1-E_\eta\bigl(e^{-\gl H_v}\bigr)\Bigr).
\end{align*}
By lemma~\ref{lem_4_2} this term converges to zero as $\eta$ converges to $v$ along
$l$. Therefore $R_\gl f$ is also continuous at $v$.

Next we prove that  for all $\gl>0$, $f\in \Co$, $R_\gl f$ vanishes at infinity. If
$\cG$ has no external edges there is nothing to prove, and so we assume that
$e\in\cE$ is an external edge of $\cG$ which is incident with the vertex $v\in V$,
$\p(e)=\{v\}$. Let $\gl$, $\gep>0$ be given. We choose $r_1\ge 0$ large enough so that
for all $\xi\in e$ with $d(v,\xi)>r_1$ we have $|f(\xi)|< \gep\gl/2$. Choose $r_2>r_1$,
and consider $\xi\in e$ with $d(v,\xi)\ge r_2$. Denote by $\xi_1$ the point on  $e$ which
has distance $r_1$ to $v$. Then we have that $P_\xi$--a.s., $H_{\xi_1}\le H_v$, and
consequently $P_\xi(H_{\xi_1}<+\infty)=1$. Hence we can use the first passage time
formula~\eqref{eq_2_6} in the form
\begin{equation*}
    R_\gl f(\xi) = E_\xi\Bigl(\int_0^{H_{\xi_1}} e^{-\gl t} f(X_t)\,dt\Bigr)
                    +E_\xi\bigl(e^{-\gl H_{\xi_1}}\bigr)\,R_\gl f(\xi_1).
\end{equation*}
For $t\in [0,H_{\xi_1}]$ we have $d(v,X_t)\ge r_1$, and therefore the absolute value
of the first term on the right hand side is bounded from above by $\gep/2$. For the
second term we can compute the expectation as for the corresponding expression of
the standard Brownian motion $B$ on $\R$, and we obtain (again with~\eqref{eq_2_5})
\begin{align*}
    \Bigl|E_\xi\bigl(e^{-\gl H_{\xi_1}}\bigr)\,R_\gl f(\xi_1)\Bigr|
        &=   e^{-\sgl\,d(\xi,\xi_1)}\,\bigl|R_\gl f(\xi_1)\bigr|\\
        &\le e^{-\sgl\,(r_2-r_1)}\,\frac{\|f\|}{\gl}.
\end{align*}
Now choose $r_2$ large enough so as to make the last term less than $\gep/2$, and
we are done. Thus we have shown that for every $\gl>0$, $R_\gl$ maps $\Co$ into
itself.

Finally, consider for $f\in \Co$, $\xi\in \cG$, $t>0$,
\begin{equation*}
    U_t f(\xi) = E_\xi\bigl(f(X_t)\bigr).
\end{equation*}
By definition, $X$ has right continuous sample paths, and $P_\xi(X_0=\xi)=1$. Since $f$
is continuous and bounded, an application of the dominated convergence theorem shows
that $U_t f(\xi)$ converges to $f(\xi)$ as $t$ decreases to $0$.
\end{proof}

\section{Generators and Feller's Theorem}     \label{sect_5}
Let $V_\cL$ denote the subset of $V\times\cL$ given by
\begin{equation*}
    V_\cL = \bigl\{(v,l),\,v\in V \text{ and }l\in\cL(v)\bigr\}.
\end{equation*}
We shall also write $v_l$ for $(v,l)\in V_\cL$. We remark in passing that
\begin{equation*}
    \bigl|V_\cL\bigr| = |\cE| + 2\,|\cI|.
\end{equation*}

Consider a real valued function $f$ on $\cG$, let $v\in V$ and let $l\in\cL(v)$
be an edge incident with $v$. We define the \emph{directional derivative of $f$
at $v$ in direction $l\in\cL(v)$} as follows:
\begin{equation}    \label{eq5i}
    f'(v_l) = \begin{cases}
                \displaystyle
                \phantom{-}\lim_{\xi\to v,\,\xi\in l\op} f'(\xi),
                        & \text{if $v$ is an initial vertex of $l$},\\[2ex]
                \displaystyle
                -\lim_{\xi\to v,\,\xi\in l\op} f'(\xi),
                        & \text{if $v$ is a final vertex of $l$},
              \end{cases}
\end{equation}
whenever the corresponding limit on the right hand side exists. Geometrically this
directional derivative is just the inward normal derivative which makes it an
intrinsic definition, independent of the orientation chosen on the edge.

\begin{definition}  \label{def_5_1}
$\Coo$ denotes the subspace of functions $f$ in $\Co$ which are twice continuously
differentiable on $\cGo$, and such that for every $v\in V$ and all $l\in\cL$
the limit
\begin{equation}   \label{eq5ia}
    f''(v_l) = \lim_{\xi\to v,\, \xi\in l\op} f''(\xi)
\end{equation}
exists. $\Cii$ denotes the subspace of those functions $f$ in $\Coo$ so that $f''$
extends from $\cGo$ to a continuous function on $\cG$.
\end{definition}

Thus $\Cii$ consists of all $f\in\Coo$ so that for every $v\in V$, the $f''(v_l)$,
$l\in\cL(v)$, are all equal.
Assume that $f\in\Cii$, and let $v\in V$. The continuous extension of $f''$ to $v$
will simply be denoted by $f''(v)$. Consider an edge $l\in\cL(v)$ incident with $v$.
Then it is easy to see that $f'(v_l)$ exists (and is finite). On the other hand, in
general for $l$, $l'\in \cL(v)$, $l\ne l'$, we have $f'(v_l)\ne f'(v_{l'})$. In
other words, in general $f'$ does \emph{not} have a continuous extension from $\cGo$
to $\cG$. Also, it is not hard to check that $f'$ vanishes at infinity.

The proof of the following lemma can be taken over with minor modifications from
the standard literature, e.g., from~\cite[Chapter~6.1]{Kn81}.

\begin{lemma}   \label{lem_5_2}
For every Brownian motion $X$ on the metric graph $\cG$, the generator $A$ of its
semigroup $U$ acting on $\Co$ has a domain $\cD(A)$ contained in $\Cii$. Moreover,
for every $f\in\cD(A)$, $A f=1/2\,f''$.
\end{lemma}

Consider data of the following form
\begin{equation}    \label{eq5ii}
\begin{split}
    a &= (a_v,\,v\in V)\in [0,1)^V\\
    b &= (b_{v_l},\,v_l\in V_\cL) \in [0,1]^{V_\cL}\\
    c &= (c_v,\,v\in V)\in [0,1]^V
\end{split}
\end{equation}
subject to the condition
\begin{equation}    \label{eq5iii}
    a_v + \sum_{l\in \cL(v)} b_{v_l} + c_v =1,\qquad \text{for every $v\in V$}.
\end{equation}
We define a subspace $\cH_{a,b,c}$ of $\Cii$ as the space of those functions
$f$ in $\Cii$ which at every vertex $v\in V$ satisfy the Wentzell boundary
condition
\begin{equation}\label{eq5iv}
    a_v f(v) - \sum_{l\in\cL(v)} b_{v_l}f'(v_l)+ \frac{1}{2}\,c_v f''(v)=0.
\end{equation}

Now we can state and prove the analogue of \emph{Feller's theorem} for metric
graphs.

\begin{theorem} \label{thm_5_3}
Suppose that $X$ is a Brownian motion on a metric graph $\cG$, and that $\cD(A)$
is the domain of the generator $A$ of its semigroup. Then there are $a$, $b$, $c$
as in~\eqref{eq5ii}, \eqref{eq5iii}, so that $\cD(A)=\cH_{a,b,c}$.
\end{theorem}

\begin{remark}  \label{rem_5_4}
The case $a_v=1$, $v\in V$, would correspond to (zero) Dirichlet boundary
conditions at the vertex $v$. The paths of the process associated with this boundary
condition have to jump instantaneously to $\gD$ when reaching the vertex, and by our
requirement that the paths are right continuous this means that the process will
never be at the vertex. But this is in contradiction to our assumption (cf.\
definition~\ref{def_3_1}) that the process with absorption at the vertex is
equivalent to a Brownian motion with absorption in the endpoint (endpoints, resp.)
of the corresponding interval. Therefore this stochastic process is \emph{not} a
Brownian motion on $\cG$ in the sense of definition~\ref{def_3_1}, and this case has
to be excluded from our discussion. Also note that in this case the semigroup does
not act strongly continuously on $\Co$, and therefore is in particular not Feller.
\end{remark}

The proof of theorem~\ref{thm_5_3} has two rather distinct parts, and therefore
we split it by proving the following two lemmas:

\begin{lemma}   \label{lem_5_5}
Suppose that $X$ is a Brownian motion on a metric graph $\cG$, and that $\cD(A)$
is the domain of the generator $A$ of its semigroup. Then there are $a$, $b$, $c$
as in~\eqref{eq5ii}, \eqref{eq5iii}, so that $\cD(A)\subset\cH_{a,b,c}$.
\end{lemma}

\begin{lemma}   \label{lem_5_6}
Suppose that $A$ is the generator of a Brownian motion $X$ on $\cG$ with domain
$\cD(A)\subset \cH_{a,b,c}$ for some $a$, $b$, $c$ as in~\eqref{eq5ii},
\eqref{eq5iii}. Then $\cD(A)=\cH_{a,b,c}$
\end{lemma}

\begin{proof}[Proof of lemma~\ref{lem_5_5}]
Our proof follows the one in~\cite[Chapter~6.1]{Kn81} quite closely --- actually, it
is sufficient to consider a special case of the proof given there.

We show that for every vertex $v\in V$ there are constants $a_v\in [0,1)$,
$b_{v_l}\in [0,1]$, $l\in\cL(v)$, $c_v\in [0,1]$ satisfying~\eqref{eq5iii},
and such that all $f$ in the domain $\cD(A)$ of the generator satisfy the boundary
condition~\eqref{eq5iv}. To this end, we let $f\in\cD(A)$, fix a vertex $v\in V$,
and compute $A f(v)$. Let us consider the three cases for $\gb$ mentioned in
section~\ref{sect_3}.

If $\gb=0$, $v$ is a trap, and $U_t f(v) = f(v)$ for all $t\ge 0$. Consequently,
$A f(v)=0$, and therefore $1/2\, f''(v)=0$. Thus $f$ satisfies the boundary
condition~\eqref{eq5iv} at $v$ with $a_v=0$, $c_v=1$, and $b_{v_l}=0$ for all
$l\in\cL(v)$.

Next we consider the case where $\gb\in (0,+\infty)$, i.e., $v$ is exponentially
holding. We know from the discussion in section~\ref{sect_3} that then after
expiration of the holding time the process jumps directly to the cemetery state.
Therefore we get for $t>0$, $U_t f(v) = \exp(-\gb t) f(v)$, and thus $A f(v) + \gb
f(v) = 0$, and the boundary condition~\eqref{eq5iv} holds for the choice
\begin{equation}    \label{eq_5_2}
   a_v = \frac{\gb}{1+\gb},\quad c_v=\frac{1}{1+\gb},\quad
            b_{v_l} = 0,\,l\in\cL(v).
\end{equation}

Finally we consider the case that $\gb=+\infty$, i.e., the process leaves $v$
immediately, and in particular, $v$ is not a trap. Therefore we may compute $A f(v)$
in Dynkin's form, e.g., \cite[p.~140, ff.]{Dy65a}, \cite[p.~99]{ItMc74}. As
in~\eqref{eq_4_3} we let $H_{v,\gep}$ denote the hitting time of the complement of
$B_\gep(v)$. Then
\begin{equation}    \label{eq_5_3}
   A f(v) = \limep \frac{E_v\Bigl(f\bigl(X(H_{v,\gep})\bigr)\Bigr)
                                                -f(v)}{E_v(H_{v,\gep})}.
\end{equation}
Recall the notation $f_l(\gep)$ for $f(\xi)$ with $\xi\in\cG$ having local coordinates
$(l,\gep)$, $l\in\cL$, $\gep\in I_l$. Then we get
\begin{align*}
    E_v\Bigl(f\bigl(X(H_{v,\gep})\bigr)\Bigr)
        &= \sum_{l\in\cL(v)} f_l(\gep)\,P_v\bigl(X(H_{v,\gep})\in l\bigr)
            + f(\gD)\,P_v\bigl(X(H_{v,\gep})=\gD\bigr)\nonumber\\
        &= \sum_{l\in\cL(v)} f_l(\gep)\,P_v\bigl(X(H_{v,\gep})\in l\bigr),
\end{align*}
where the last equality follows from $f(\gD)=0$. Let us denote
\begin{align*}
    r_l(\gep)   &= \frac{P_v\bigl(X(H_{v,\gep})\in l\bigr)}{E_v(H_{v,\gep})},
                        \qquad l\in\cL(v),\\[1ex]
    r_\gD(\gep) &= \frac{P_v\bigl(X(H_{v,\gep})=\gD\bigr)}{E_v(H_{v,\gep})},\\[1ex]
    K(\gep)     &= 1 + r_\gD(\gep) + \gep \sum_{l\in\cL(v)} r_l(\gep).
\end{align*}
The continuity of the paths of $X$ up to the lifetime $\zeta$ yields
\begin{equation*}
    \sum_{l\in\cL(v)} P_v\bigl(X(H_{v,\gep})\in l\bigr)
                            + P_v\bigl(X(H_{v,\gep})=\gD\bigr)=1,
\end{equation*}
and therefore equation~\eqref{eq_5_3} can be rewritten as
\begin{equation*}
    \limep\Bigl(A f(v) + r_\gD(\gep) f(v)
        - \sum_{l\in\cL(v)} r_l(\gep) \bigl(f_l(\gep)-f(v)\bigr)\Bigr)=0.
\end{equation*}
Since for all $\gep>0$, $K(\gep)^{-1}\le 1$, it follows that
\begin{equation*}
    \limep\Bigl(\frac{1}{K(\gep)}\,A f(v) + \frac{r_\gD(\gep)}{K(\gep)}\,f(v)
        - \sum_{l\in\cL(v)} \frac{\gep\, r_l(\gep)}{K(\gep)}\,
                \frac{f_l(\gep)-f(v)}{\gep}\Bigr)=0,
\end{equation*}
which by lemma~\ref{lem_5_2} we may rewrite as
\begin{equation*}
    \limep\Bigl(a_v(\gep) f(v) + \frac{1}{2}\,c_v(\gep) f''(v)
        - \sum_{l\in\cL(v)} b_{v_l}(\gep)\,\frac{f_l(\gep)-f(v)}{\gep}\Bigr)=0,
\end{equation*}
where we have introduced the non-negative quantities
\begin{align*}
    a_v(\gep)       &= \frac{r_\gD(\gep)}{K(\gep)},\\
    c_v(\gep)       &= \frac{1}{K(\gep)},\\
    b_{v_l}(\gep)   &= \frac{\gep\, r_l(\gep)}{K(\gep)},\qquad l\in\cL(v).
\end{align*}
Observe that for every $\gep>0$,
\begin{equation*}
    a_v(\gep) + c_v(\gep) + \sum_{l\in\cL(v)} b_{v_l}(\gep) =1.
\end{equation*}
Therefore every sequence $(\gep_n,\,n\in\N)$ with $\gep_n>0$ and $\gep_n\da 0$ has a
subsequence so that $a_v(\gep)$, $c_v(\gep)$ and $b_{v_l}(\gep)$, $l\in\cL(v)$,
converge along this subsequence to numbers $a_v$, $c_v$, and $b_{v_l}$ respectively in
$[0,1]$, and the relation~\eqref{eq5iii} holds true. From the remark after
definition~\ref{def_5_1} it follows that
\begin{equation*}
    \frac{f_l(\gep)-f(v)}{\gep}
\end{equation*}
converges with $\gep\da 0$ to $f'(v_l)$, and therefore we obtain that for every
vertex $v\in V$, $f\in\cD(A)$ satisfies the boundary condition~\eqref{eq5iv} with data
$a$, $b$, $c$ as in~\eqref{eq5ii}, \eqref{eq5iii}.
\end{proof}

Before we can prove lemma~\ref{lem_5_6} we have to introduce some additional
formalism.

For given data $a$, $b$, $c$ as in~\eqref{eq5ii}, \eqref{eq5iii}, it will be convenient
to consider $\cH_{a,b,c}$ equivalently as being the subspace of
$\Coo$ so that for its elements $f$ at every $v\in V$ the boundary
conditions~\eqref{eq5iv} as well as the boundary condition
\begin{equation}    \label{eq5iva}
    f''(v_l) = f''(v_k),\qquad     \text{for all $l,\,k\in\cL(v)$}
\end{equation}
hold true. Relation~\eqref{eq5iva} is just another way to express that $f''$ is
continuous on $\cG$.

We consider the sets $V$, $\cE$, and $\cI$ as being ordered in some arbitrary
way. With the convention that in $\cL$ the elements of $\cE$ come first this
induces also an order relation on $\cL$.

Suppose that $f\in\Coo$. With the given ordering of $\cE$ and $\cI$ we define the
following column vectors of length $|\cE|+2|\cI|$:
\begin{align*}
    f(V)    &= \Bigl(\bigl(f_e(0),\,e\in\cE\bigr),\bigl(f_i(0),\,i\in\cI\bigr),
                                        \bigl(f_i(\rho_i)\,i\in\cI\bigr)\Bigr)^t,\\
    f'(V)   &= \Bigl(\bigl(f'_e(0),\,e\in\cE\bigr),\bigl(f'_i(0),\,i\in\cI\bigr),
                                        \bigl(-f'_i(\rho_i)\,i\in\cI\bigr)\Bigr)^t,\\
    f''(V)  &= \Bigl(\bigl(f''_e(0),\,e\in\cE\bigr),\bigl(f''_i(0),\,i\in\cI\bigr),
                                        \bigl(f''_i(\rho_i)\,i\in\cI\bigr)\Bigr)^t,
\end{align*}
where the superscript ``$t$'' indicates transposition.

We want to write the boundary conditions~\eqref{eq5iv}, \eqref{eq5iva} in a compact
way, and to this end we introduce the following order relation on $V_\cL$: For $v_l$,
$v'_{l'}\in V_\cL$ we set $v_l \preceq v'_{l'}$ if and only if $v \prec v'$ or $v =
v'$ and $l\preceq l'$ (where for $V$ and $\cL$ we use the order relations introduced
above). For $f$ as above set
\begin{align*}
    \tilde f(V)   &= \bigl(f(v_l),\,v_l\in V_\cL\bigr)^t,\\
    \tilde f'(V)  &= \bigl(f'(v_l),\,v_l\in V_\cL\bigr)^t,\\
    \tilde f''(V) &= \bigl(f''(v_l),\,v_l\in V_\cL\bigr)^t.
\end{align*}
Then there exists a permutation matrix $P$ so that
\begin{equation*}
    \tilde f(V) = P f(V),\qquad \tilde f'(V) = P f'V),\qquad \tilde f''(V) = P f''(V).
\end{equation*}
In particular, $P$ is an orthogonal matrix which has in every row and in every
column exactly one entry equal to one while all other entries are zero.

For every $v\in V$ we define the following $|\cL(v)|\times|\cL(v)|$ matrices:
\begin{align*}
    \tilde A(v) &= \begin{pmatrix}
                    a_v    & 0      & 0      & \cdots & 0\\
                    0      & 0      & 0      & \cdots & 0\\
                    \vdots & \vdots & \vdots & \ddots & \vdots\\
                    0      & 0      & 0      & \cdots & 0
                 \end{pmatrix},\\[2ex]
    \tilde B(v) &= \begin{pmatrix}
                    -b_{v_{l_1}} & -b_{v_{l_2}} & -b_{v_{l_3}} & \cdots & -b_{v_{l_{|\cL(v)|}}}\\
                    0            & 0            & 0            & \cdots & 0\\
                    \vdots       & \vdots       & \vdots       & \ddots & \vdots\\
                    0            & 0            & 0            & \cdots & 0
                 \end{pmatrix},\\[2ex]
    \tilde C(v) &= \begin{pmatrix}
                    1/2\,c_v & 0      & 0      & 0      & \cdots & 0      \\
                    1        & -1     & 0      & 0      & \cdots & 0      \\
                    0        & 1      & -1     & 0      & \cdots & 0      \\
                    0        & 0      & 1      & -1     & \cdots & 0      \\
                    \vdots   & \vdots & \ddots & \ddots & \ddots & \vdots \\
                    0        & 0      & 0      & 0      & \cdots & -1
                 \end{pmatrix},
\end{align*}
where we have labeled the elements in $\cL(v)$ in such a way that in the above
defined ordering we have $l_1\prec l_2 \prec \dotsb \prec l_{|\cL(v)|}$. Observe
that $\tilde C(v)$ is invertible if and only if $c_v\ne 0$. Define block matrices
$\tilde A$, $\tilde B$, and $\tilde C$ by
\begin{equation*}
    \tilde A = \bigoplus_{v\in V} A(v),\quad \tilde B = \bigoplus_{v\in V} B(v),
            \quad \tilde C = \bigoplus_{v\in V} C(v).
\end{equation*}
Then we can write the boundary conditions~\eqref{eq5iv}, \eqref{eq5iva} simultaneously for
all vertices as
\begin{equation}    \label{bch}
    \tilde A \tilde f(V) + \tilde B \tilde f'(V) + \tilde C \tilde f''(V) = 0.
\end{equation}
Consequently the boundary conditions can equivalently be written in the form
\begin{equation}    \label{bc}
    Af(V) + Bf'(V) + Cf''(V) = 0,
\end{equation}
with
\begin{equation}
    A = P^{-1}\tilde A P,\quad B = P^{-1}\tilde B P,\quad C = P^{-1}\tilde C P.
\end{equation}

We bring in the following two matrix-valued functions on the complex plane
\begin{equation}    \label{Zpm}
    \hat Z_\pm (\gk) = A \pm \gk B + \gk^2 C,\qquad \gk\in\C.
\end{equation}

\begin{lemma}   \label{lem_5_7}
There exists $R>0$ so that for all $\gk\in\C$ with $|\gk|\ge R$ the matrices $\hat
Z_\pm(\gk)$ are invertible, and there are constants $C$, $p>0$ so that
\begin{equation}    \label{invnorm}
    \|\hat Z_\pm(\gk)^{-1}\|\le C\,|\gk|^p ,\qquad |\gk|\ge R.
\end{equation}
\end{lemma}

\begin{remark}  \label{rem_5_8}
The bound~\eqref{invnorm} is actually rather crude, but sufficient for our purposes.
\end{remark}

\begin{proof}[Proof of lemma~\ref{lem_5_7}]
Since we have
\begin{equation}    \label{ZP}
    \hat Z_{\pm}(\gk) = P^{-1}\bigl(\tilde A \pm\gk \tilde B + \gk^2\tilde C\bigr) P
\end{equation}
for an orthogonal matrix $P$, for the proof of the first statement it suffices
to show that there exists $R>0$ such that
\begin{equation*}
    \tilde A \pm\gk \tilde B + \gk^2\tilde C
\end{equation*}
are invertible for complex $\gk$ outside of the open ball of radius $R$. For this in turn
it suffices to show that for every vertex $v\in V$ the matrices
\begin{equation*}
\begin{split}
    \tilde A(v) &\pm\gk \tilde B(v) + \gk^2\tilde C(v)\\
        &= \begin{pmatrix}
                    a_v \pm \gk b_{v_{l_1}}+ \gk^2/2\,c_v & \pm\gk b_{v_{l_2}} & \pm\gk b_{v_{l_3}}
                             & \pm\gk b_{v_{l_4}} & \cdots & \pm\gk b_{v_{l_{|\cL(v)|}}}      \\
                    \gk^2    & -\gk^2 & 0      & 0      & \cdots & 0      \\
                    0        & \gk^2  & -\gk^2 & 0      & \cdots & 0      \\
                    0        & 0      & \gk^2  & -\gk^2 & \cdots & 0      \\
                    \vdots   & \vdots & \vdots & \ddots & \ddots & \vdots \\
                    0        & 0      & 0      & 0      & \cdots & -\gk^2
          \end{pmatrix}
\end{split}
\end{equation*}
are invertible for all $\gk\in\C$ with $|\gk|\ge R$. An elementary calculation
gives
\begin{equation*}
    \det\bigl(\tilde A(v) \pm\gk \tilde B(v) + \gk^2\tilde C(v)\bigr)
        = \Bigl(a_v \pm \gk \sum_{l\in \cL(v)}b_{v_l} + \frac{\gk^2}{2}\,c_v\Bigr)
                \bigl(-\gk^2\bigr)^{|\cL(v)|-1}.
\end{equation*}
The choices $\gk=\pm 1$ together with condition~\eqref{eq5iii} show that the polynomial
of second order in $\gk$ in the first factor on the right hand side does not vanish identically.
Therefore, it is non-zero in the exterior of an open ball with some radius $R_v>0$.
Hence, we obtain the first statement for the choice $R = \max_{v\in V} R_v$.
Moreover, from the calculation of the determinants above we also get for every $v\in V$ and
all $\gk\in\C$ with $|\gk|\ge R$ an estimate of the form
\begin{equation}    \label{detinv}
    \bigl|\det\bigl(\tilde A(v) \pm\gk \tilde B(v) + \gk^2\tilde C(v)\bigr)\bigr|^{-1}
        \le \text{const.}
\end{equation}
Thus, using the co-factor formula for
\begin{equation*}
    \bigl(\tilde A(v) \pm\gk \tilde B(v) + \gk^2\tilde C(v)\bigr)^{-1}
\end{equation*}
we find with~\eqref{detinv} the estimate
\begin{equation*}
    \bigl\|\bigl(\tilde A(v) \pm\gk \tilde B(v) + \gk^2\tilde C(v)\bigr)^{-1}\bigr\|
        \le C_v |\gk|^{p_v},\qquad |\gk|\ge R,
\end{equation*}
for some constants $C_v$, $p_v>0$. Consequently we get
\begin{equation*}
    \bigl\|\bigl(\tilde A \pm\gk \tilde B + \gk^2\tilde C\bigr)^{-1}\bigr\|
        \le C |\gk|^p,\qquad |\gk|\ge R,
\end{equation*}
for some constants $C$, $p>0$, and by~\eqref{ZP} we have proved inequality~\eqref{invnorm}.
\end{proof}

With these preparations we can enter the

\begin{proof}[Proof of lemma~\ref{lem_5_6}]
Let the data $a$, $b$, $c$ be given as in\eqref{eq5ii}, \eqref{eq5iii}. We have to
show that the inclusion $\cD(A)\subset \cH_{a,b,c}$ is not strict. Let $R =
(R_\gl,\,\gl>0)$ be the resolvent of $A$. Then for every $\gl>0$, $R_\gl$
is a bijection from $C_0(\cG)$ onto $\cD(A)$, that is, $R_\gl^{-1}$ is a bijection from
$\cD(A)$ onto $C_0(\cG)$.

Assume to the contrary that the inclusion  $\cD(A)\subset \cH_{a,b,c}$ is strict.
We will derive a contradiction. For $\gl>0$ consider the linear mapping
$H_\gl:\,f\mapsto \gl f - 1/2 f''$
from $\cH_{a,b,c}$ to $C_0(\cG)$. On $\cD(A)$ this mapping coincides with $R^{-1}_\gl$, and
$R^{-1}_\gl$ is a bijection from $\cD(A)$ onto $C_0(\cG)$. Therefore our
assumption entails that $H_\gl$ cannot be injective. Hence for any $\gl>0$ there exists
$f(\gl)\in\cH_{a,b,c}$, $f(\gl)\ne 0$, with
\begin{equation}    \label{homeq}
    H_\gl f(\gl) = \gl f(\gl) - \frac{1}{2}\,f''(\gl) = 0.
\end{equation}
We will show that $f(\gl)\in\cH_{a,b,c}$ satisfying \eqref{homeq} can only hold when
$f(\gl)=0$ on $\cG$. It will be convenient to change the
variable $\gl$ to $\gk = \sqrt{2\gl}$, and there will be no danger of confusion that
we shall simply write $f(\gk)$ for $f(\gl)$ from now on. Then the solution
of~\eqref{homeq} is necessarily of the form given by
\begin{align}
    f_e(\gk,x) &= r_e(\gk)\,e^{-\gk x}  & e&\in\cE,\,x\in \R_+,\\
    f_i(\gk,x) &= r^+_i(\gk)\,e^{\gk x}+r^-_i(\gk)\,e^{\gk(\rho_i- x)}
                                        & i&\in\cI,\,x\in [0,\rho_i],
\end{align}
and we want to show that for some $\gk > 0$, the boundary conditions~\eqref{eq5iv}
and~\eqref{eq5iva} entail that $r_e(\gk) = r^+_i(\gk) = r^-_i(\gk)=0$ for all
$e\in\cE$, $i\in\cI$. For $\gk>0$, define a column vector $r(\gk)$ of length
$|\cE|+2|\cI|$ by
\begin{equation*}
    r(\gk) = \bigl((r_e(\gk),\,e\in\cE),(r^+_i(\gk),\,i\in\cI),(r^-_i(\gk),\,i\in\cI)\bigr)^t,
\end{equation*}
and introduce the $(|\cE|+2|\cI|)\times (|\cE|+2|\cI|)$ matrices
\begin{equation*}
    X_\pm(\gk)   = \begin{pmatrix}
                    \1 & 0                & 0                \\
                    0  & \1               & \pm e^{\gk \rho} \\
                    0  & \pm e^{\gk \rho} & \1
                   \end{pmatrix}
\end{equation*}
--- appropriately modified in case that $\cE$ or $\cI$ is the empty set --- with
the $|\cI|\times|\cI|$ diagonal matrices
\begin{equation*}
    e^{\gk \rho} = \diag{e^{\gk \rho_i},\,i\in\cI\bigr}.
\end{equation*}
Then the boundary conditions~\eqref{eq5iv}, \eqref{eq5iva} for $f(\gk)$ read
\begin{equation}    \label{Zr}
    Z(\gk)r(\gk) = 0,
\end{equation}
with
\begin{equation}    \label{defZ}
    Z(\gk) = (A+\gk^2C)X_+(\gk) + \gk BX_-(\gk).
\end{equation}
Thus, if we can show that for some $\gk > 0$ the matrix $Z(\gk)$ is invertible, the
proof of the theorem is finished. Note that the matrix-valued function $Z$ is
entire in $\gk$, and therefore so is its determinant. Thus, if can show that $\gk\mapsto \det
Z(\gk)$ does not vanish identically, then it can only vanish on a discrete subset of
the complex plane, and for $\gk$ in the complement of this set $Z(\gk)$ is
invertible. Write
\begin{equation*}
    X_\pm(\gk) = \1 \pm \gd X(\gk),
\end{equation*}
with
\begin{equation*}
    \gd X(\gk) = \begin{pmatrix}
                    0  & 0                & 0            \\
                    0  & 0                & e^{\gk \rho} \\
                    0  & e^{\gk \rho}     & 0
                   \end{pmatrix},
\end{equation*}
so that we can write
\begin{equation*}
    Z(\gk) = \hat Z_+(\gk)\bigl(\1+\gd Z(\gk)\bigr),
\end{equation*}
with
\begin{equation*}
    \gd Z(\gk) = \hat Z_+(\gk)^{-1}\, \hat Z_-(\gk)\,\gd X(\gk).
\end{equation*}
Observe that in case that $\cI=\emptyset$, we obtain $\gd Z(\gk)=0$, and in this
case the invertibility of $Z(\gk)$ for all $\gk$ with $\gk\ge R$ follows from
lemma~\ref{lem_5_7}. Hence we assume from now on that $\cI\ne \emptyset$.
Lemma~\ref{lem_5_7} provides us with the bound
\begin{equation*}
    \bigl\|\hat Z_+(\gk)^{-1}\hat Z_-(\gk)\bigr\| \le \text{const.}\,|\gk|^{q},
\end{equation*}
for all $\gk\in\C$, $|\gk|\ge R$, and for some $q>0$. On the other hand, we get
\begin{equation*}
    \|\gd X(\gk)\| \le e^{\gk \rho_0},
\end{equation*}
for all $\gk\le 0$ where $\rho_0 = \min_{i\in\cI} \rho_i$. Therefore, there exists a
constant $R'>0$ so that for all $\gk\le -R'$  we have $\|\gd Z(\gk)\|<1$, and therefore
for such $\gk$, $Z(\gk)$ is invertible, i.e., $\det Z(\gk)\ne 0$. Hence there also
exists $\gk > 0$ so that $Z(\gk)$ is invertible, and the proof is finished.
\end{proof}

\begin{appendix}
\section{Feller Semigroups and Resolvents}  \label{app_A}
In this appendix we give an account of the Feller property of semigroups and
resolvents. The material here seems to be quite well-known, and our presentation of
it owes very much to~\cite{Ra56}, most notably the inversion formula for the Laplace
transform, equation~\eqref{inv_L} in connection with lemma~\ref{lem_A_6}. On the
other hand, we were not able to locate a reference where the results are collected and
stated in the form in which we employ them in the present paper. Therefore we also
provide proofs for some of the statements.

Assume that $(E,d)$ is a locally compact separable metric space with Borel
$\gs$--algebra denoted by $\cB(E)$. $B(E)$ denotes the space of bounded measurable
real valued functions on $E$, $\Coe$ the subspace of continuous functions vanishing
at infinity. $B(E)$ and $\Coe$ are equipped with the sup-norm $\|\,\cdot\,\|$.

The following definition is as in~\cite{ReYo91}:
\begin{definition}  \label{def_A_1}
    A \emph{Feller semigroup} is a family $U=(U_t,\,t\ge 0)$ of positive
    linear operators on $\Coe$ such that
    \begin{enum_i}
        \item $U_0=\text{id}$ and $\|U_t\|\le 1$ for every $t\ge 0$;
        \item $U_{t+s} = U_t\comp U_s$ for every pair $s$, $t\ge 0$;
        \item $\lim_{t\downarrow 0} \|U_t f - f\|=0$ for every $f\in \Coe$.
    \end{enum_i}
\end{definition}

Analogously we define
\begin{definition}  \label{def_A_2}
    A \emph{Feller resolvent} is a family $R=(R_\gl,\,\gl>0)$ of positive
    linear operators on $\Coe$ such that
    \begin{enum_i}
        \item $\|R_\gl\|\le \gl^{-1}$ for every $\gl>0$;
        \item $R_\gl - R_\mu = (\mu-\gl) R_\gl\comp R_\mu$ for every
                pair $\gl$, $\mu>0$;
        \item $\lim_{\gl\to\infty}\|\gl R_\gl f - f\|=0$ for every $f\in\Coe$.
    \end{enum_i}
\end{definition}

In the sequel we shall focus our attention on semigroups $U$ and resolvents $R$
associated with an $E$--valued Markov process, and which are \emph{a priori} defined
on $B(E)$. (In our notation, we shall not distinguish between $U$ and $R$ as defined
on $B(E)$ and their restrictions to $\Coe$.)

Let $X = (X_t,\,t\ge 0)$ be a Markov process with state space $E$, and let
$(P_x,\,x\in E)$ denote the associated family of probability measures on some measurable
space $(\gO,\cA)$ so that $P_x(X_0=x) = 1$. $E_x(\,\cdot\,)$ denotes the
expectation with respect to $P_x$. We assume throughout that for every $f\in B(E)$
the mapping
\begin{equation*}
    (t,x) \mapsto E_x\bigl(f(X_t)\bigr)
\end{equation*}
is measurable from $\R_+\times E$ into $\R$. The semigroup $U$ and resolvent $R$
associated with $X$ act on $B(E)$ as follows. For $f\in B(E)$, $x\in E$, $t\ge 0$,
and $\gl>0$ set
\begin{align}
    U_t f(x)   &= E_x\bigl(f(X_t)\bigr),  \label{eq_A_1}\\
    R_\gl f(x) &= \int_0^\infty e^{-\gl t} U_t f (x)\,dt.  \label{eq_A_2}
\end{align}
Property~(i) of Definitions~\ref{def_A_1} and \ref{def_A_2} is obviously satisfied. The
semigroup property, (ii) in Definition~\ref{def_A_1}, follows from the Markov property
of $X$, and this in turn implies the resolvent equation, (ii) of
Definition~\ref{def_A_2}. Moreover, it follows also from the Markov property of $X$
that the semigroup and the resolvent commute. On the other hand, in general neither
the property that $U$ or $R$ map $\Coe$ into itself, nor the strong continuity
property (iii) in Definitions~\ref{def_A_1}, \ref{def_A_2} hold true on $B(E)$ or on
$\Coe$.

If $W$ is a subspace of $B(E)$ the resolvent equation shows that the image of $W$
under $R_\gl$ is independent of the choice of $\gl>0$, and in the sequel we shall
denote the image by $RW$. Furthermore, for simplicity we shall write $U\Coe\subset
\Coe$, if $U_t f\in\Coe$ for all $t\ge 0$, $f\in\Coe$.

\begin{theorem} \label{thm_A_3}
The following statements are equivalent:
\begin{enum_i}
    \item $U$ is Feller.
    \item $R$ is Feller.
    \item $U\Coe\subset\Coe$, and for all $f\in\Coe$, $x\in E$,
            $\lim_{t\downarrow 0} U_t f(x) = f(x)$.
    \item $U\Coe\subset\Coe$, and for all $f\in\Coe$, $x\in E$,
            $\lim_{\gl\rightarrow \infty} \gl R_\gl f(x) = f(x)$.
    \item $R\Coe\subset\Coe$, and for all $f\in\Coe$, $x\in E$,
            $\lim_{t\downarrow 0} U_t f(x) = f(x)$.
    \item $R\Coe\subset\Coe$, and for all $f\in\Coe$, $x\in E$,
            $\lim_{\gl\rightarrow \infty} \gl R_\gl f(x) = f(x)$.
\end{enum_i}
\end{theorem}

We prepare a sequence of lemmas. The first one follows directly from the
dominated convergence theorem:

\begin{lemma}   \label{lem_A_4}
Assume that for $f\in B(E)$, $U_t f \rightarrow f$ as $t\downarrow 0$. Then $\gl
R_\gl f \rightarrow f$ as $\gl\to+\infty$.
\end{lemma}

\begin{lemma}   \label{lem_A_5}
The semigroup $U$ is strongly continuous on $RB(E)$.
\end{lemma}

\begin{proof}
If strong continuity at $t=0$ has been shown, strong continuity at $t>0$ follows
from the semigroup property of $U$, and the fact that $U$ and $R$ commute. Therefore
it is enough to show strong continuity at $t=0$.

Let $f\in B(E)$, $\gl>0$, $t>0$, and consider for $x\in E$ the following computation
\begin{align*}
    U_t R_\gl f(x) &- R_\gl f(x)\\
        &= \int_0^\infty e^{-\gl s} E_x\bigl(f(X_{t+s})\bigr)\,ds
                - \int_0^\infty e^{-\gl s} E_x\bigl(f(X_s)\bigr)\,ds\\
        &= e^{\gl t} \int_t^\infty e^{-\gl s} E_x\bigl(f(X_s)\bigr)\,ds
                - \int_0^\infty e^{-\gl s} E_x\bigl(f(X_s)\bigr)\,ds\\
        &= \bigl(e^{\gl t}-1\bigr) \int_t^\infty e^{-\gl s} E_x\bigl(f(X_s)\bigr)\,ds
                - \int_0^t e^{-\gl s} E_x\bigl(f(X_s)\bigr)\,ds\\
\end{align*}
where we used Fubini's theorem and the Markov property of $X$. Thus we get the
following estimation
\begin{align*}
    \bigl\|U_t R_\gl f - R_\gl f\bigr\|
        &\le \biggl(\bigl(e^{\gl t} - 1\bigr)\int_t^\infty e^{-\gl s}\,ds
                +\int_0^t e^{-\gl s}\,ds\biggl)\, \|f\|\\
        &= \frac{2}{\gl}\, \bigl(1-e^{-\gl t}\bigr)\,\|f\|,
\end{align*}
which converges to zero as $t$ decreases to zero.
\end{proof}

For $\gl>0$, $t\ge 0$, $f\in B(E)$, $x\in E$ set
\begin{equation}    \label{inv_L}
    U^\gl_t f(x) = \sum_{n=1}^\infty \frac{(-1)^{n+1}}{n!}
                        \,n\gl\, e^{n\gl t}\, R_{n\gl} f(x).
\end{equation}
Observe that, because of $n\gl\|R_{n\gl} f\| \le \|f\|$, the last sum converges in
$B(E)$.

For the proof of the next lemma we refer the reader
to~\cite[p.~477~f]{Ra56}:

\begin{lemma}   \label{lem_A_6}
For all $t\ge 0$, $f\in RB(E)$, $U^\gl_t f$ converges in $B(E)$ to $U_t f$ as
$\gl$ tends to infinity.
\end{lemma}

\begin{lemma}   \label{lem_A_7}
If $U_t\Coe \subset \Coe$ for all $t\ge 0$, then $R_\gl\Coe\subset \Coe$, for all $\gl>0$.
If $R_\gl\Coe\subset \Coe$, for some $\gl>0$, and $R_\gl\Coe$ is dense in $\Coe$, then
$U_t\Coe \subset \Coe$ for all $t\ge 0$.
\end{lemma}

\begin{proof}
Assume that  $U_t\Coe \subset \Coe$ for all $t\ge 0$, let $f\in\Coe$, $x\in E$, and
suppose that $(x_n,\,n\in\N)$ is a sequence converging in $(E,d)$ to $x$. Then a
straightforward application of the dominated convergence theorem shows that for
every $\gl>0$, $R_\gl f(x_n)$ converges to $R_\gl f(x)$. Hence $R_\gl f\in \Coe$.

Now assume that that $R_\gl\Coe\subset \Coe$, for some and therefore for all
$\gl>0$, and that $R_\gl\Coe$ is dense in $\Coe$. Consider $f\in R\Coe$, $t>0$, and
for $\gl>0$ define $U^\gl_t f$ as in equation~\eqref{inv_L}. Because
$R_{n\gl}f\in\Coe$ and the series in formula~\eqref{inv_L} converges uniformly in
$x\in E$, we get $U^\gl_t f\in\Coe$. By lemma~\ref{lem_A_6}, we find that $U^\gl_t
f$ converges uniformly to $U_t f$ as $\gl\to+\infty$. Hence $U_t f\in\Coe$. Since
$R\Coe$ is dense in $\Coe$, $U_t$ is a contraction and $\Coe$ is closed, we get that
$U_t\Coe\subset\Coe$ for every $t\ge 0$.
\end{proof}

The following lemma is proved as a part of Theorem~17.4 in~\cite{Ka97}
(cf.\ also the proof of Proposition~2.4 in~\cite{ReYo91}).

\begin{lemma}   \label{lem_A_8}
Assume that $R\Coe\subset \Coe$, and that for all $x\in E$, $f\in\Coe$,
$\lim_{\gl\to\infty} \gl R_\gl f(x) = f(x)$. Then $R\Coe$ is dense in $\Coe$.
\end{lemma}

If for all $f\in\Coe$, $x\in E$, $U_t f(x)$ converges to $f(x)$ as $t$ decreases to
zero, then similarly as in the proof of lemma~\ref{lem_A_4} we get that $\gl R_\gl
f(x)$ converges to $f(x)$ as $\gl\to+\infty$. Thus we obtain the following

\begin{corollary}   \label{cor_A_9}
Assume that $R\Coe\subset \Coe$, and that for all $x\in E$, $f\in\Coe$,
$\lim_{t\downarrow 0} U_t f(x) = f(x)$. Then $R\Coe$ is dense in $\Coe$.
\end{corollary}

Now we can come to the

\begin{proof}[Proof of theorem~\ref{thm_A_3}]
We show first the equivalence of statements~(i), (ii), (iv), and (vi):

\vspace{.5\baselineskip}\noindent
``(i)\Ra(ii)'' Assume that $U$ is Feller. From lemma~\ref{lem_A_7} it follows that
$R_\gl\Coe\subset\Coe$, $\gl>0$. Let $f\in\Coe$. Since $U$ is strongly continuous on
$\Coe$, lemma~\ref{lem_A_4} implies that $\gl R_\gl f$ converges to $f$ as $\gl$ tends
to $+\infty$. Hence $R$ is Feller.

\vspace{.25\baselineskip}\noindent
``(ii)\Ra(vi)'' This is trivial.

\vspace{.25\baselineskip}\noindent
``(vi)\Ra(iv)'' By lemma~\ref{lem_A_8}, $R\Coe$ is dense in $\Coe$, and therefore
lemma~\ref{lem_A_7} entails that $U\Coe\subset\Coe$.

\vspace{.25\baselineskip}\noindent
``(iv)\Ra(i)'' By lemmas~\ref{lem_A_7} and \ref{lem_A_8},  $R\Coe$ is dense in
$\Coe$, and therefore by lemma~\ref{lem_A_5} $U$ is strongly continuous on $\Coe$. Thus
$U$ is Feller.

\vspace{.25\baselineskip}
Now we prove the equivalence of~(i), (iii), and (v):

\vspace{.25\baselineskip}\noindent
``(i)\Ra(iii)'' This is trivial.

\vspace{.25\baselineskip}\noindent
``(iii)\Ra(v)'' This follows directly from Lemma~\ref{lem_A_7}.

\vspace{.25\baselineskip}\noindent
``(v)\Ra(i)'' By corollary~\ref{cor_A_9}, $R\Coe$ is dense in $\Coe$, hence it follows
from lem\-ma~\ref{lem_A_7} that $U\Coe\subset\Coe$. Furthermore, lemma~\ref{lem_A_5}
implies the strong continuity of $U$ on $R\Coe$, and by density therefore on $\Coe$.
(i) follows.
\end{proof}

\end{appendix}
\providecommand{\bysame}{\leavevmode\hbox to3em{\hrulefill}\thinspace}
\providecommand{\MR}{\relax\ifhmode\unskip\space\fi MR }
\providecommand{\MRhref}[2]{%
  \href{http://www.ams.org/mathscinet-getitem?mr=#1}{#2}
}
\providecommand{\href}[2]{#2}

\end{document}